\documentclass[11pt]{amsart}

\usepackage{amscd}
\usepackage{stmaryrd}
\usepackage{amsmath, amssymb}
\usepackage{amsfonts}

\newcommand{\ddbar}{\sqrt{-1} \partial \overline{\partial}}

\newcommand{\ti}[1]{\widetilde{#1}}
\newcommand{\abs}[1]{\lvert#1\rvert}
\newcommand{\vp}{\varphi}
\newcommand{\vol}{\mathrm{Vol}}

\newcommand{\ve}{\varepsilon}

\renewcommand{\leq}{\leqslant}
\renewcommand{\geq}{\geqslant}

\newcommand{\be}{\begin{equation}}
\newcommand{\ee}{\end{equation}}

\newcommand{\PSH}{\mathrm{PSH}}

\newcommand{\CP}{\mathbb{CP}}
\newcommand{\D}{\mathbb D}
\newcommand{\C}{\mathbb C}
\newcommand{\e}{\varepsilon}
\newcommand{\tr}{\mathrm{tr}}
\begin{document}
\newtheorem{claim}{Claim}
\newtheorem{theorem}{Theorem}[section]
\newtheorem{lemma}[theorem]{Lemma}
\newtheorem{corollary}[theorem]{Corollary}
\newtheorem{proposition}[theorem]{Proposition}
\newtheorem{question}{question}[section]
\newtheorem{defn}{Definition}[theorem]
\theoremstyle{definition}
\newtheorem{remark}[theorem]{Remark}

\numberwithin{equation}{section}

\title[Envelopes with prescribed singularities]{Envelopes with prescribed singularities}
\begin{abstract}We prove that quasi-plurisubharmonic envelopes with prescribed analytic singularities in suitable big cohomology classes on compact K\"ahler manifolds have the optimal $C^{1,1}$ regularity on a Zariski open set. This also proves regularity of certain pluricomplex Green's functions on K\"ahler manifolds. We then go on to prove the same regularity for envelopes when the manifold is assumed to have boundary. As an application, we answer affirmatively a question of Ross--Witt-Nystr\"om concerning the Hele-Shaw flow on an arbitrary Reimann surface.
\end{abstract}
\author[N. McCleerey]{Nicholas McCleerey}
\address{Department of Mathematics, Northwestern University, 2033 Sheridan Road, Evanston, IL 60208}
\email{njm2@math.northwestern.edu}

\thanks{Partially supported by NSF RTG grant DMS-1502632.}

\maketitle

\section{Introduction}

Let $X^n$ be a compact K\"ahler manifold without boundary and $M^n$ a compact K\"ahler manifold with boundary. In this note, we will be concerned with proving regularity of envelopes of the following two types:
\begin{equation}\label{env}
P_{[\psi]}(u)(x)=\sup\{v(x)\ |\ v\in \textrm{PSH}(X,\theta), v\leq u, v\leq \psi+C, \textrm{ for some }C\}^*
\end{equation}
and
\begin{align}\label{bddenv}
V_{[\psi]}(u)(x)=\sup\{v(x)\ |\ v\in \textrm{PSH}(M,\omega), &\limsup_{z\rightarrow z_0} v(z) \leq u(z_0)\forall z_0\in\partial M,\\ & v\leq \psi+C, \textrm{ for some }C\}^*.\nonumber
\end{align}

Here $\theta$ is a closed, smooth, real $(1,1)$-form on $X$ such that the cohomology class $[\theta]$ is big, $\omega$ is a K\"ahler form on $M$, $u$ is a smooth function on $X$/$M$ (the ``obstacle'' of the envelope), and $\psi$ is a quasi-plurisubharmonic function with analytic singularities (the ``singularity type'' of the envelope) whose singularities do not intersect the boundary. Here and in the following $f^*$ means the upper semi-continuous regularization of $f$. We will also fix another closed, real $(1,1$)-form $\gamma$ such that $\psi$ is $\gamma$-psh.

Envelopes of the form \eqref{env} first appeared in the work of Ross-Witt Nystr\"om \cite{RWN,RWN2} in relation with geodesic rays in the space of K\"ahler metrics (see also the survey \cite{RWN3}), and they have played a central role in recent developments, see e.g. \cite{Da,DDL,DDL2} and references therein. Since $\psi$ has analytic singularities, if $\theta=\gamma$ then it is easy to see (cf. \cite{RS}, \cite[Remark 3.9]{RWN2}) that the envelope $P_{[\psi]}(u)$ differs from $\psi$ by a bounded amount, i.e. it has the same singularity type as $\psi$.

If we choose $\psi$ a bounded function, then \eqref{env} reduces to the envelope
\begin{equation}\label{iguess}
P(u)(x)=\sup\{\vp(x)\ |\ \vp\in \textrm{PSH}(X,\theta), \vp\leq u\},
\end{equation}
which was originally used as an example of a $\theta$-psh function with ``minimal singularities'' (\cite{DPS}), and \eqref{bddenv} reduces to the classical Perron-Bremermann envelopes \cite{Br}. The regularity of both has been studied extensively -- in particular, Berman \cite{Be3,Be,Be2} and Berman-Demailly \cite{BD} showed that $P(u)$ has bounded Laplacian, and hence is in $C^{1,\alpha}$ for all $\alpha < 1$ on compact subsets of the ample locus of the class $[\theta]$ (a Zariski open subset of $X$). Examples in \cite[Example 5.2]{Be3}, \cite[Example 7.2]{CT}, and \cite[Corollary 4.7]{SZ} show that one cannot in general expect $P(u)\in C^2$, and so the best possible regularity is $C^{1,1}$, which was indeed shown by Tosatti \cite{To} and Chu-Zhou \cite{CZ} when $[\theta]$ is K\"ahler. More recently, $C^{1,1}$ regularity of $P(u)$ on compact subsets of the ample locus of $[\alpha]$ has been proved by Chu-Tosatti-Weinkove \cite{CTW1} when $[\alpha]$ is nef and big. These results are obtained by combining the ``zero temperature limit'' deformation of Berman \cite{Be} with the $C^{1,1}$ estimates for degenerate complex Monge-Amp\`ere equations developed by Chu-Tosatti-Weinkove in \cite{CTW,CTW2}.

The envelope \eqref{bddenv} with $u=0$ is also known as the ``pluricomplex Green's function,'' as introduced in \cite{Kl} (see also \cite{Le}) with $\psi(z)=\log|z|^2$ and in \cite{RS} for general $\psi$ with analytic singularities, and whose regularity has been studied extensively by many authors \cite{BedDem,Bl,Gu,PS1}, primarily in the case when $\psi$ has poles at a discreet set of points. For \eqref{env}, the paper of Ross-Witt Nystr\"om \cite{RWN2} proved that if $[\theta]=c_1(L)$ and $[\gamma]=c_1(F)$ are the first Chern classes of holomorphic line bundles $L$ and $F$ over $X$ with $L-F$ big, and if $e^\psi$ is H\"older continuous (which is more general than $\psi$ having analytic singularities) then \eqref{env} has the optimal $C^{1,1}$ regularity on compact subsets of the ample locus $L-F$ away from the singularities of $\psi$. For this, they adapted the technique of Berman \cite{Be3}, which is in turn inspired by a classical $C^{1,1}$ estimate of Bedford-Taylor \cite{BT}.

We now state our results. The first is concerning regularity of the envelope \eqref{env}:
\begin{theorem}\label{thrm}
Let $X$ be a compact K\"ahler manifold and $\theta, \gamma$ two closed real $(1,1)$-forms on $X$ such that $[\theta]$ is big, $[\gamma]$ is pseudoeffective, and $[\theta-\gamma]$ is nef. Suppose also that $u\in C^\infty(X)$ and $\psi\in PSH(X,\gamma)$ with analytic singularities. Then if either:
\begin{enumerate}
\item $[\theta - \gamma]$ is big, or
\item $\int_X \langle (\gamma +  \ddbar \psi)^n\rangle > 0,$
\end{enumerate}
we have that $P_{[\psi]}(u)\in C^{1,1}_{loc}(X\setminus Y)$, where $Y\subset X$ is a proper analytic subset given either by $Y=E_{nK}(\theta-\gamma)\cup\mathrm{Sing}(\psi)$ in case $(1)$, or by $Y=\mu(E_{nK}(\ti{\gamma}))\cup\mathrm{Sing}(\psi)$ in case $(2)$ (where $\ti{\gamma}$ and $\mu$ are defined in \eqref{dec2} below).
\end{theorem}

Here $\langle\cdot\rangle$ is the non-pluripolar product of \cite{BEGZ}. Case $(1)$ of this theorem is an extension of \cite[Theorem 1.1]{RWN2} to transcendental cohomology classes, in the special case when $\psi$ has analytic singularities and the class $[\theta - \gamma]$ is nef. As a special example of case $(2)$ one can for example take $\theta=\gamma$.

We remark that the $C^{1,1}$ regularity of the envelope \eqref{env} in general was explicitly raised as an open problem in \cite[p.2]{CZ}, which our result solves in this setting. As mentioned above, this also gives the optimal regularity of pluricomplex Green's functions of the form \eqref{env} with $u=0$, with ``pole type'' given by $\psi$.

Also, as in \cite[Section 4.5]{RWN2} and \cite{Be3,BD,CTW1,To}, we obtain as a direct corollary that:
\begin{equation}\label{snoop}
\chi_{X\backslash Y}(\theta+\ddbar P_{[\psi]}(u))^n=\chi_{\{P_{[\psi]}(u)=u\}}(\theta+\ddbar u)^n.
\end{equation}
We also see that, in the case that $X$ is of complex dimension one, the boundary of the ``contact set" $\{P_{[\psi]}(u) = u\}$ has measure zero -- see \cite[Cor. 5.13]{RWN3}. Interestingly, the proof does not generalize to higher dimensions, as $C^{1,1}$ regularity is not enough to guarantee the existance of a foliation of holomorphic disks, even locally \cite{DG}.

Theorem \ref{thrm} will be deduced from \cite[Theorem 1.3]{CTW1} after some reductions that will allow us to transform the envelope \eqref{env} into one of the form \eqref{iguess} on a bimeromorphic model of $X$. In particular, our method of proof is quite different from the one in \cite{RWN2} in the case of line bundles. The assumption that $[\theta - \gamma]$ be nef is only needed so that we may apply \cite[Theorem 1.3]{CTW1} -- if one can remove this assumption from \cite[Theorem 1.3]{CTW1}, then it can also be dropped from Theorem \ref{thrm} without changing the proof.

Our second set of results concern the envelopes \eqref{bddenv}, and using them we are able to answer a question of Ross--Witt-Nystr\"om concerning regularity of an envelope appearing in their work on the Hele-Shaw flow on a Riemann surface \cite[Def. 6.1]{RWN}. Suppose that $X$ is now a closed Riemann surface, and let $\D\subset \C$ be the open unit disk and $\overline{\D}^* := \overline{\D}\setminus\{0\}$ the punctured unit disk. Let $\omega$ be a K\"ahler form on $X$ with $\int_X \omega = 1$ and $\pi: X\times\D\rightarrow X$ be the projection map. Then we have:

\begin{theorem}\label{bad}
Let $ 0 \leq t\leq 1$ and $z_0\in X$. Define the following envelopes:
\begin{equation}\label{HS}
\vp_t := \sup\left\{ v\in\PSH(X\times\overline{\D}, \pi^*\omega)\ \middle|\ \limsup_{|\tau|\rightarrow 1} v(z,\tau)\leq 0,\ \ \nu(v, (z_0, 0)) \geq t\right\}.
\end{equation}
Then if  $t < 1$, we have that $\vp_t\in C^{1,1}_\text{loc}((X\times\overline{\D})\setminus\{(z_0, 0)\})$. If $t =1$, we have that $\vp_1\in C^{1,1}_\text{loc}(X\times\overline{\D}^*)$.
\end{theorem}

Indeed, the envelope $\vp_1$ is precisely the envelope \cite[Def. 6.1]{RWN}. In \cite[Rmk. 6.11]{RWN}, the authors speculate that $\vp_1$ should be in $C^{1,\alpha}_\text{loc}$ for all $\alpha < 1$, and actually prove that it is in $C^{1,1}_\text{loc}$ when $X = \CP^1$, using either the results of \cite{CTW2} or \cite{Bl2}. We thus show that this stronger result indeed always holds. As a corollary, one can now apply \cite[Cor 6.7]{RWN} to the Hele-Shaw flow on a general Riemann surface $X$ to see that it is always strictly increasing at at least some fixed positive rate.

Proving Theorem \ref{bad} breaks into two cases: the case $t < 1$ follows a similar argument to Theorem \ref{thrm} and ends up reducing to the regularity result in \cite{CTW2}. Similar reductions in the $t=1$ case do not however yield the situation in \cite{CTW2}, and so we use the following theorem, which is a varient of the estimates in \cite{CTW1} and \cite{Bou}:

\begin{theorem}\label{weak}Suppose that $(M,\omega)$ is a compact K\"ahler manifold with boundary. Suppose moreover that $\alpha$ is a closed, real $(1,1)$-form on $M$ such that the following holds:
\begin{itemize}
\item There exists a function $\psi\in\PSH(M,\alpha)$,  $\psi\leq 0$, with analytic singularities, smooth near and up to the boundary of $M$, that also satisfies $\alpha + \ddbar\psi\geq \delta\omega$ for some small $\delta> 0$.
\item For every $\ve > 0$, there exists a smooth function $u_\ve\leq -1$ such that $\alpha_\ve := \alpha + \e\omega + \ddbar u_\ve > 0$ is a K\"ahler form.
\item There exists a function $v_0\in\PSH(M,\alpha)$ such that $v_0\leq 0$, $v_0|_{\partial M}\equiv 0$, and $v_\e := \ti{\max}\{v_0,u_\e\}$ is a smooth $(\alpha+\ve\omega)$-psh function up to the boundary for all $\ve > 0$.
\end{itemize}
Then the envelope:
\begin{equation}\label{thethingprime}
V := \sup\left\{ v\in\PSH(M, \alpha)\ \middle|\ \limsup_{|z|\rightarrow \partial M} v(z)\leq 0\right\}
\end{equation}
is uniformly Lipschitz on compact subsets of $M\setminus\mathrm{Sing}(\psi)$.

If we further have that $\alpha + \ddbar v_0\geq \delta\omega$ in a neighborhood of $\partial M$ and that the boundary of $M$ is pseudoconcave, then actually $V\in C^{1,1}_\text{loc}(M\setminus\mathrm{Sing}(\psi))$.
\end{theorem}

One could replace the boundary obstacle 0 in the above theorem with any smooth function on $M$ in the obvious manner.

A few words are now in order about the function $v_0$. It is well-known that on a K\"ahler manifold with boundary, it may not always be possible to solve the Dirichlet problem for the complex Monge-Ampere equation. It is always possible to solve if one however assumes the existance of a smooth (strong) subsolution \cite{Bou}. Our $v_0$ then plays the role of this subsolution. However, since we are not assuming any sort of semi-positivity on $\alpha$, one cannot expect there to exist any smooth $\alpha$-psh functions -- in particular, $v_0$ will be unbounded in general. We then require smooth-ness of $v_0$ away from its unbounded locus, which is satisfied if, for example, $v_0$ has analytic singularities. Finally, it is unclear to us if the assumption that $v_0$ be strictly $\alpha$-psh near $\partial M$ is necessary for the Hessian estimates or not -- at any rate, it is enough for our application, and so we do not address this here.

As a concluding remark, we would like to point out that many envelopes of the form \eqref{bddenv} can be dealt with using the techniques in this paper, if one assumes additional information relating $\omega$ and $\psi$ or additional information on $M$; indeed, Theorem \ref{bad} is an example of such a result. There does not appear to be a precise framework one can invoke to cover all such cases, but the general theme is that $\psi$ have relatively ``mild'' analytic singularities, which can usually be checked cohomologically/geometrically.\\

This paper is organized as follows. In section two, we give the proof of Theorem \ref{thrm}. In section three, we prove Theorem \ref{bad} by reducing the envelopes \eqref{HS} to either the situation in \cite{CTW2} (when $t < 1$) or that of Theorem \ref{weak}, which is then proved in section four. We then give an application to global extremal functions, which may be well-known to experts (Remark \ref{extremal}). \\

\noindent
{\bf Acknowledgments. } We would like to thank Julius Ross for useful comments about the Hele-Shaw flow and for encouraging us to prove Theorem \ref{bad}. I would also like to thank my advisor Valentino Tosatti for introducing me to this circle of questions, for greatly improving the format of this paper, and for his continued patience and guidance.

\section{Proof of Theorem \ref{thrm}}
\begin{proof}[Proof of Theorem \ref{thrm}]
Let us note right away that it is enough to prove Theorem \ref{thrm} with $u=0$, since we can write
\[\begin{split}
&P_{[\psi]}(u)(x)-u(x)=\\
&\sup\{v(x)\ |\ v\in \textrm{PSH}(X,\theta+\ddbar u), v\leq 0, v\leq \psi+C, \textrm{ for some }C\}^*.
\end{split}
\]
In the following we thus let $u=0$, and denote the corresponding envelope in \eqref{env} simply by $w$.

Recall now that we say that $\psi$ has analytic singularities if there exists a real number $c > 0$ and a coherent ideal sheaf $\mathcal J$ such that for every $x\in X$, there is a small neighborhood $U\subseteq X$ of $x$ so that $\mathcal J$ is generated on $U$ by holomorphic functions $(f_1,\ldots, f_N)$ and:
\[
\psi|_U = c\log\left(\sum_{i=1}^N|f_i|^2\right) + h,
\]
for some smooth, real valued function $h$, defined locally on $U$. We can then apply Hironaka's theorem \cite{Hir} to resolve $\mathcal J$ (thus resolving the singularities of $\psi$ also), giving us $\mu:\ti{X}\to X$, with $\mu$ a composition of blowups with smooth centers, $\ti{X}$ a compact K\"ahler manifold, and
\begin{equation}\label{dec2}
\mu^*(\gamma + \ddbar\psi) = \ti{\gamma} + \llbracket E\rrbracket,
\end{equation}
where $E=\sum_i\lambda_i E_i$ is an effective $\mathbb{R}$-divisor ($\lambda_i\in\mathbb{R}_{>0}$) with $\mathrm{Supp}(E)=\mathrm{Exc}(\mu)$, $\llbracket E\rrbracket$ is its current of integration, and $\ti{\gamma}$ is a closed, semipositive $(1,1)$-form, which, by definition of the non-pluripolar product, satisfies
\[
\int_{\ti{X}}\ti{\gamma}^n=\int_X \langle(\gamma+\ddbar \psi)^n\rangle.
\]
If (2) is satisfied, then this number is positive, so that the class $[\ti\gamma]$ is both nef (being semipositive) and big.

Fix now defining sections $s_i$ of $\mathcal{O}(E_i)$ and smooth metrics $h_i$ on $\mathcal{O}(E_i)$ with curvature form $R_{h_i}$. For brevity, denote
$$|s|^2_h=\prod_i |s_i|^{2\lambda_i}_{h_i}, \quad R_h=\sum_i\lambda_i R_{h_i}.$$
Then by the Poincar\'e-Lelong formula
\[
\llbracket E\rrbracket=\ddbar\log|s|_h+R_h,
\]
so that:
\begin{equation}\label{dec}
[\mu^*\theta - R_h] = \mu^*[\theta - \gamma] + [\ti\gamma].
\end{equation}
Nefness is preserved under pull-back by $\mu$, and since the nef cone is convex, we see that $[\mu^*\theta - R_h]$ is nef. Now if (1) is satisfied, then $\mu^*[\theta-\gamma]$ is big, as volume is preserved under $\mu^*$ as well. If (2) is satisfied, then as we have already remarked, $[\ti\gamma]$ is big. Thus, in either case, at least one of the two classes is big, and so their sum $[\mu^*\theta - R_h]$ must also be big (and nef).

Note also that by definition of analytic singularities, the function $\mu^*\psi - \log|s|_h$ is in fact smooth on all of $\ti{X}$.

Next, note that in case $(1)$ we can use \eqref{dec} to see that
$$E_{nK}(\mu^*\theta-R_h)\subseteq E_{nK}(\mu^*(\theta-\gamma)),$$
while
$$E_{nK}(\mu^*(\theta-\gamma))=\mu^{-1}(E_{nK}(\theta-\gamma))\cup\mathrm{Exc}(\mu),$$
by \cite[Proposition 2.5]{To2}, and since $\mu(\mathrm{Exc}(\mu))\subseteq\mathrm{Sing}(\psi)$, we conclude that
$$\mu(E_{nK}(\mu^*\theta-R_h))\subseteq E_{nK}(\theta-\gamma)\cup \mathrm{Sing}(\psi).$$
In case $(2)$ we again use \eqref{dec} to get
$$E_{nK}(\mu^*\theta-R_h)\subseteq E_{nK}(\ti{\gamma}),$$
so that in either case we have $\mu(E_{nK}(\mu^*\theta-R_h)\cup \mathrm{Supp}(E))\subseteq Y$, where $Y$ is as in the statement of Theorem \ref{thrm}.

We shall now use an idea of Berman \cite[Section 4]{Be}. Note that
\[
\mu^*w(x)=\sup\{v(x)\ |\ v\in \textrm{PSH}(\ti{X},\mu^*\theta), v\leq 0, v\leq \mu^*\psi+C, \textrm{ for some }C\}^*,
\]
because every $\mu^*\theta$-psh function on $\ti{X}$ is in fact the pullback of a $\theta$-psh function on $X$ (since $\mu$ has connected fibers). Since $\mu^*\psi - \log|s|_h$ is smooth, we also have
\[
\mu^*w(x)=\sup\{v(x)\ |\ v\in \textrm{PSH}(\ti{X},\mu^*\theta), v\leq 0, v\leq \log|s|_h+C, \textrm{ for some }C\}^*,
\]
We claim that we actually have
\begin{equation}\label{key}
\mu^*w(x)=\log|s|_h(x)+\sup\{v(x)\ |\ v\in PSH(\ti{X},\mu^*\theta-R_h), v\leq -\log|s|_h\}.
\end{equation}
Note that here we can drop the upper-semi-continuous regularization from the envelope, since its upper-semi-continuous regularization is a candidate for the supremum.
To see our claim, denote the right hand side by $U$. For one direction, if $v$ is $(\mu^*\theta-R_h)$-psh and satisfies $v\leq -\log|s|_h$, then
$\ti{v}:=v+\log|s|_h$ satisfies $\ti{v}\leq 0$. But since we also have $v\leq C$ on $\ti{X}$, it follows that $\ti{v}\leq \log|s|_h+C$, and
\[
\begin{split}
\mu^*\theta+\ddbar \ti{v}&=\mu^*\theta+\ddbar\log|s|_h+\ddbar v\\
&=\mu^*\theta-R_h+\llbracket E\rrbracket+\ddbar v\\
&\geq \mu^*\theta-R_h+\ddbar v\geq 0,
\end{split}
\]
so that $U\leq\mu^*w$. Conversely, if $\ti{v}$ is $\mu^*\theta$-psh and satisfies $\ti{v}\leq 0$ and $\ti{v}\leq \log|s|_h+C$ for some $C$, then the Siu decomposition of $\mu^*\theta+\ddbar \ti{v}$ contains $\llbracket E\rrbracket$ and so
$$0\leq \mu^*\theta+\ddbar \ti{v}-\llbracket E\rrbracket=\mu^*\theta-R_h+\ddbar(\ti{v}-\log|s|_h).$$
Thus $v:=\ti{v}-\log|s|_h$ is $(\mu^*\theta-R_h)$-psh and satisfies $v\leq -\log|s|_h,$ and it follows that $\mu^*w\leq U$, which proves \eqref{key}.

Using \eqref{key}, since the term $\log|s|_h$ is smooth on $\ti{X}\backslash \mathrm{Supp}(E)$, it is enough to prove that the envelope
$$V(x)=\sup\{v(x)\ |\ v\in PSH(\ti{X},\mu^*\theta-R_h), v\leq -\log|s|_h\},$$
satisfies $V\in C^{1,1}_{\rm loc}(\ti{X}\backslash E_{nK}(\mu^*\theta-R_h))$. To do this, we show that we can actually replace the unbounded obstacle $-\log|s|_h$ with a (smooth)  bounded obstacle. Having done this, we can appeal directly to \cite[Theorem 1.3]{CTW1} to conclude the proof, after again changing our reference form to reduce to the case when the obstacle is uniformly 0.

Start by fixing a small neighborhood $U$ of $\mathrm{Supp}(E)$, and let $$M:=\sup_{\ti{X}\setminus U} (-\log|s|_h) <\infty.$$ Recall then that the (global) extremal function of $\ti{X}\setminus U$ is defined by
\[
V_{\ti{X}\setminus U}(x) := \sup\{v(x)\ |\ v\in PSH(\ti{X},\mu^*\theta-R_h), v\leq 0 \text{ on } \ti{X}\setminus U\}.
\]
By the remarks in \cite[Section 4]{BEGZ}, $\sup_{\ti{X}} V_{\ti{X}\setminus U} < \infty$ if and only if $X\setminus U$ is not pluripolar. We can easily arrainge this however, e.g. by picking $U$ such that $\ti{X}\setminus U$ has positive (Lebesgue) measure. By our choice of $U$ then,
$$
(V - M)|_{\ti{X}\setminus U}\leq (-\log|s|_h - M)|_{\ti{X}\setminus U} \leq 0
$$
and so
\[
V \leq M + V_{\ti{X}\setminus U}\leq M + \sup_{\ti{X}} V_{\ti{X}\setminus U} =: \tilde{M} < \infty,
\]
holds on $\ti{X}$.
Consequently,
\begin{equation}\label{pp}
V\leq \min\{-\log|s|_h, \ti M\}.
\end{equation}
Denote by $\widetilde{\min}$ a regularized minimum function with error $1/2$; i.e. on the set $\{|u - v| > 1/2\}$, we have $\widetilde{\min}\{u,v\} = \min\{u,v\}$, and when $|u - v| \leq 1/2$, $\widetilde{\min}\{u,v\}$ smoothly interpolates between $u$ and $v$. Thus, we always have the two-sided bound:
\begin{equation}\label{min}
\min\{ u -1/2, v - 1/2\}\leq \widetilde{\min}\{u,v\}\leq\min\{u,v\}.
\end{equation}
Explicitly, we can take $\widetilde{\min}(u,v)=-\widetilde{\max}(-u,-v)$, where $\widetilde{\max}$ is the regularized maximum constructed in \cite[Lemma 5.18]{De}.

Then we claim that we actually have
\begin{equation}\label{ahh}
V(x) = \sup\left\{v(x)\ \middle|\ v\in PSH(\ti{X},\mu^*\theta-R_h), v\leq \widetilde{\min}\{-\log|s|_h, \ti M + 1\}\right\},
\end{equation}
which will finish the proof (by using \cite[Theorem 1.3]{CTW1}), since now the obstacle $$\widetilde{\min}\{-\log|s|_h, \ti M + 1\}$$ is smooth on $\ti{X}$.

To prove \eqref{ahh}, note that if $v\in PSH(\ti{X},\mu^*\theta-R_h)$ satisfies
$$v\leq \widetilde{\min}\{-\log|s|_h, \ti M + 1\},$$
then in particular, by \eqref{min}, $v\leq -\log|s|_h$. If on the other hand $v$ satisfies $v\leq -\log|s|_h$,
then by definition $v\leq V$ and using \eqref{pp} we have
$$v\leq V\leq \min\{-\log|s|_h, \ti M\}.$$
To conclude, we observe that
\[
\min\{-\log|s|_h, \ti M\}\leq  \widetilde{\min}\{-\log|s|_h, \ti M+1\}.
\]
This inequality is obvious if $|\ti M + 1 + \log|s|_h| > 1/2$, while if on the other hand $|\ti M + 1 + \log|s|_h| \leq 1/2$, then
\[
\ti M + 1 + \log|s|_h \leq 1/2 \implies \ti M \leq -\log|s|_h - 1/2,
\]
and so, using \eqref{min} again,
\[\begin{split}
\min\{-\log|s|_h, \ti M\}&=\ti{M}=\min\{-\log|s|_h-1/2,\ti{M}\}\\
&\leq \widetilde{\min}\{-\log|s|_h,\ti{M}+1/2\}\leq \widetilde{\min}\{-\log|s|_h,\ti{M}+1\},
\end{split}\]
which concludes the proof.
\end{proof}

\section{An Application to Hele-Shaw Flow}

We will now show regularity of the envelopes $\vp_t$. We will actually prove a slightly stronger theorem, in that we do not need $X$ to be a Reimann surface:
\begin{theorem}\label{badprime}
Suppose that $(X^n,\omega)$ is a compact K\"ahler manifold without boundary and $\D\subset\C$ is the unit disk. Let $z_0\in X$ and $\pi:X\times\overline{\D}\rightarrow X$ be the projection map. Define for all $0\leq t\leq \e_{z_0}(\omega)$ the envelope:
\[
\vp_t := \sup\left\{ \psi\in\PSH(X\times\overline{\D}, \pi^*\omega)\ \middle|\ \limsup_{|\tau|\rightarrow 1} \psi(z,\tau)\leq 0, \nu(\psi, (z_0, 0)) \geq t\right\}.
\]
Then for $t <  \e_{z_0}(\omega)$, we have that $\vp_t\in C^{1,1}_\text{loc}((X\times\overline{\D})\setminus\{(z_0, 0)\})$. If $t = \e_{z_0}(\omega)$, we have that $\vp_1\in C^{1,1}_\text{loc}(X\times\overline{\D}^*)$. Here $\e_{z_0}(\omega)$ is the Seshadri constant of $[\omega]$ at $z_0$.
\end{theorem}

Theorem \ref{bad} follows from Theorem \ref{badprime} by using the easy fact that, if $X$ is a Riemann surface, then $\e_{z_0}(\omega) = \vol(\omega) = \int_X\omega$.

\begin{remark}
In Theorem \ref{badprime}, one can easily replace $\D\subset\C$ with the unit ball $B(1)\subset\C^n$ with only minimal changes to the proof. More generally, one can use any smooth, bounded, hyperconvex domain $D\subset\C^N$ which admits a smooth strictly plurisubharmonic exhaustion function $u$ such that $\ddbar u$  is actually the restricion of a global, real $(1,1)$-form $\beta$ on $\CP^N$ such that $[\beta]$ is big.
\end{remark}

In what follows, we will sometimes think of $X\times\overline{\mathbb{D}}$ as being a subset of $X\times\CP^1$, using the standard inclusion and stereographic projection, $\mathbb D\subset \mathbb C\subset\CP^1$. We write $p: X\times\CP^1\rightarrow\CP^1$ for the projection map onto $\CP^1$. We first recall an easy proposition on product manifolds:

\begin{proposition}\label{dumppy}
Suppose that $X$ and $Y$ are compact K\"ahler manifolds, with K\"ahler forms $\omega_X$ and $\omega_Y$, respectively. Let $\pi_X$ and $\pi_Y$ be the projections from $X\times Y$. Then for any $x\in X$ and $y\in Y$:
\[
\ve_{(x,y)}(\pi_X^*\omega_X + \pi_Y^*\omega_Y) = \min\{\ve_x(\omega_X), \ve_y(\omega_Y)\}
\]
and
\[
\zeta_{(x,y)}(\pi_X^*\omega_X + \pi^*_Y\omega_Y) \geq \zeta_x(\omega_X) + \zeta_y(\omega_Y),
\]
where $\zeta_x(\omega_X)$ is the Nakayama constant of $[\omega_X]$ at $x\in X$.
\end{proposition}
\begin{proof}

Recall the definitions for the Seshadri and Nakayama constants: Suppose that $Z$ is a compact K\"ahler manifold with $\theta$ a K\"ahler form and $z\in Z$. Let $\mu: \ti{Z}\rightarrow Z$ be the blow-up of $Z$ at $z$, with exceptional divisor $E$. Then we define
\begin{align*}
&\ve_z(\theta) := \sup\{t\geq 0\ |\ [\mu^*\theta] - t[E]\text{ is a K\"ahler class}\}\text{ and }\\
&\zeta_z(\theta) := \sup\{t\geq 0\ |\ [\mu^*\theta] - t[E]\text{ is a big class}\}.
\end{align*}
It is shown in Demailly \cite{De0}, that we have an equivalent definition for the Seshadri constant in terms of the Lelong number:
\[
\ve_z(\theta) = \sup\{ \nu(T,z)\ |\ T\geq 0, T\in[\theta],\text{ and } z \text{ is an isolated singularity of }T\},
\]
where here we will say that $z$ is an isolated singularity of $T$ if we can write $T = \theta + \ddbar\psi$ with $\psi$ continuous on $U\setminus\{z\}$, for some open neighborhood $U$ of $z$. Note that, since $\theta$ is K\"ahler, in the above definition we can also request that $T$ be a K\"ahler current, that it have analytic singularities, and that it actually be smooth on all of $Z\setminus\{z\}$, using Demailly approximation.

For ease of notation, let $\omega := \pi_X^*\omega_X + \pi^*_Y\omega_Y$ on $X\times Y$. We can prove out first claim as follows: if $T\in[\omega]$ is as above, being singular only at $(x,y)$, the restriction $T|_{\pi^{-1}_Y(y)}$ makes sense, and gives a K\"ahler current in $[\pi^*_X\omega_X|_{\pi^{-1}_Y(y)}] = [\omega_X]$ with Lelong number greater than or equal to $\nu(T, (x,y))$, as the Lelong number can only increase under restriction; this is immediate from the characterization of the Lelong numbers of $T = \omega + \ddbar\psi$ as:
\begin{equation}\label{logdef}
\nu(T, z) = \sup\{\gamma\ |\ \psi(p)\leq \gamma\log |p - z| + C,\text{ for some constant } C\},
\end{equation}
where here $|p - z|$ is the Euclidean distance in any local coordinate chart about $z$. This shows that
\[
\ve_{(x,y)}(\omega)\leq\min\{\ve_x(\omega_X),\ve_y(\omega_Y)\}.
\]

For the reverse inequality, we fix $\delta > 0$, and find K\"ahler currents $T_X\in[\omega_X]$ and $T_Y\in[\omega_Y]$ with analytic singularities only at $x$ and $y$, respectively, and having $\nu(T_X,x) \geq \ve_x(\omega_X) - \delta$ and $\nu(T_Y,y) \geq \ve_y(\omega_Y) - \delta$. Writing $T_X = \omega_X + \ddbar\psi_X$ and $T_Y = \omega_Y + \ddbar\psi_Y$, we see that the pull-backs
\[
\pi_X^*\psi_X, \pi^*_Y\psi_Y\in\PSH(X\times Y, \omega),
\]
and so we also have
\[
\psi := \max\{\pi_X^*\psi_X, \pi^*_Y\psi_Y\}\in\PSH(X\times Y, \omega).
\]
This $\psi$ is then a continuous function on $(X\times Y)\backslash\{(x,y)\}$ with an isolated singularity at $(x,y)$ and having $\nu(\psi,(x,y))\geq \min\{\ve_x(\omega_X) - \delta, \ve_y(\omega_Y) - \delta\}$, which is again clear from \eqref{logdef}. Letting $\delta\rightarrow 0$ then shows the first claim.

For the second claim, we use the following characterization of the Nakayama constant in terms of Lelong numbers \cite[Prop 3.1]{MX}:
\[
\zeta_z(\theta) = \sup\{\nu(T,z)\ |\ T\geq 0, T\in[\theta]\}.
\]
Thus, if $T_X\in[\omega_X]$ realizes $\nu(T_X,x) = \zeta_x(\omega_X)$, we have that $\pi_X^*T_X$ is a positive current on $X\times Y$ with generic Lelong number
\[
\nu(\pi_X^*T_X, \pi_X^{-1}(x)) = \nu(T_X,x),
\]
and so $\nu(\pi_X^*T_X, (x, y')) \geq \nu(T_X,x)$ for all $y'\in Y$. Doing the same for $y$, we then get that the current $T := \pi_X^*T_X  + \pi_Y^*T_Y$ is a positive current cohomologous to $\omega$ such that $\nu(T, (x,y)) \geq \nu(T_X,x) + \nu(T_Y,y)$, which proves the claim.

\end{proof}

We then get the following corollary:
\begin{corollary}\label{the frog}
Let $\omega_X,\omega_Y,$ and $\omega$ be as above, and $\mu: \mathrm{Bl}_{(x,y)}(X\times Y)\rightarrow X\times Y$ the blow-up map of $X\times Y$ at $(x,y)$ with exceptional divisor $E$. If then $\ve_x(\omega_X) < \ve_y(\omega_Y)$, we have that $[\mu^*\omega - \ve_{(x,y)}(\omega)\llbracket E\rrbracket]$ is a big and nef class. Furthermore, $E_{nK}(\mu^*\omega - \ve_{(x,y)}(\omega)\llbracket E\rrbracket) \subseteq \mu^{-1}(\pi_Y^{-1}(y))$.
\end{corollary}
\begin{proof}
The first part follows immediately from Proposition \ref{dumppy}, which tells us that $\ve_{(x,y)}(\omega) = \ve_x(\omega_X) < \ve_y(\omega_Y)\leq \zeta_{(x,y)}(\omega)$, and so the class $[\mu^*\omega - \ve_{(x,y)}(\omega)\llbracket E\rrbracket]$ must still be big, and not mearly pseudoeffective (that it is nef is always true, as the nef cone is closed).

For the second part, pull-back a K\"ahler current $T_Y\geq \delta\omega_Y$ ($\delta > 0$) from $Y$ to $X\times Y$, with $T_Y$ having an isolated singularity at $y$ with Lelong number $\ve_{(x,y)}(\omega)$; $T_Y$ exists again by Proposition \ref{dumppy}. Then $T_Y + \pi^*_X\omega_X\geq\delta\omega$ will be a K\"ahler current on $X\times Y$ that is smooth away from $\pi^{-1}_Y(y)$. Pulling back by $\mu$ gives:
\[
T := \mu^*(T_Y + \pi^*_X\omega_X)\geq \delta\mu^*\omega,
\]
where $T\in\mu^*[\omega]$ is now a positive current which is smooth and strictly positive away from $\mu^{-1}(\pi_Y^{-1}(y))$. We can then find a smooth form $R_h\in [E]$ such that $\delta\mu^*\omega - \delta'R_h$ is a K\"ahler form for all sufficently small $\delta' > 0$; thus, $T - \delta' R_h + \delta'\llbracket E\rrbracket$ is a K\"ahler current with analytic singularities, still in $\mu^*[\omega]$, with generic Lelong number along $E$ satisfying:
\[
\nu(T - \delta' R_h + \delta'\llbracket E\rrbracket, E) = \nu(T_Y, y) + \delta' > \ve_{(x,y)}(\omega).
\]
By the Siu decomposition, $T - \delta' R_h + (\delta' - \ve_{(x,y)}(\omega))\llbracket E\rrbracket$ is still a K\"ahler current, now in $[\mu^*\omega - \ve_{(x,y)}(\omega)\llbracket E\rrbracket]$, singular only on $\mu^{-1}(\pi_Y^{-1}(y))$, and so we are done.
\end{proof}

We may now prove Theorem \ref{badprime}:

\begin{proof}[Proof of Theorem \ref{badprime}]
Assume without loss of generality that $\int_X\omega^n = 1$. The following reductions are basically the same as what we have done in the previous section. We first rewrite $\vp_t$ as follows:
\begin{align*}
\vp_t = 2p^*u + \sup\{v\in&\PSH(X\times\overline{\mathbb D}, \pi^*\omega + 2\ddbar (p^*u))\ |\ \\&\limsup_{|\tau|\rightarrow 1}v(z,\tau) \leq 0,\text{ and } \nu(v,(z_0,0))\geq t\},
\end{align*}
where $u(\tau) = (1/2\pi)\log(1 + |\tau|^2) - (1/2\pi)\log(2)$ is the local potential for the Fubini-Study metric $\omega_{FS}$ on $\overline{\mathbb{D}}\subset \mathbb{CP}^1$, normalized to be 0 on $\partial\mathbb D$. The proof of this is the same as it is for the case of a total obstacle, but let us emphasize that we need as an additional assumption that $u$ be continuous near and up to the boundary, even though this is quite obvious in this case.


Letting then $\log|s|_h$ be the same as in the previous section, but on $\ti{X} := \mathrm{Bl}_{(z_0,0)}(X\times\CP^1)$ now, we again have that
\[
\ddbar\log|s|_h = \llbracket E\rrbracket - R_h,
\]
for a smooth form $R_h\in[E]$. Let $\mu$ be the blow-up map, as usual. Note that, if $U$ is a local coordinate chart around $(z_0,0)$, then
\[
\log|\mu(p) - (z_0,0)| - \log|s(p)|_h
\]
is bounded for all $p\in \mu^{-1}(U)$, where here $|\cdot|$ is the Euclidean distance in the coordinate patch $U$. Thus, using \eqref{logdef} and the fact that $\mu$ has connected fibres, one can check that:
\begin{align*}
\mu^*\vp_t = 2p^*u + \sup&\{v\in\PSH(\ti{X}, \mu^*(\pi^*\omega + 2p^*\omega_{FS}))\ |\ \\&\limsup_{|\tau|\rightarrow 1}v(z,\tau) \leq 0 \text{ and } v\leq t\log|s|_h + C,\text{ for some } C > 0\},
\end{align*}
where in the above expression we are implicitly using $(z,\tau)$ as coordinates on $\ti{X}\setminus E$, via the isomorphism $\mu$ with $(X\times\CP^1)\setminus \{(z_0,0)\}$. We can now use the same observation of Berman \cite{Be} as before to see that:
\begin{align*}
\mu^*\vp_t = 2\mu^*p^*u + &t\log|s|_h + \sup\{v\in\PSH(\ti{X}, \mu^*(\pi^*\omega + 2p^*\omega_{FS}) - tR_h)\ |\\
&\ \limsup_{|\tau| < 1, \tau\rightarrow \ti{\tau}}v(z,\tau)\leq -t\log|s(z,\ti{\tau})|_h\text{ for all }\ti{\tau}\in\partial\mathbb{D}\}.
\end{align*}
Again we emphasize that we need $\log|s|_h$ to be continuous near and up to the boundary.\\

We can now avail ourselves of Proposition \ref{dumppy}. Recall the elementary fact that if $\theta$ is a K\"ahler form on $Y$ with volume one, then $\ve_y(\theta)\leq\vol(\theta)^{1/n} = 1$ for all $y\in Y$, and that we always have equality in the case when $Y$ is a Riemann surface; see, for example, \cite[Theorem 4.6]{To2}, or the explicit construction in \cite[Lemma 5.5]{RWN3}. Thus, $\ve_\tau(2\omega_{FS}) = 2 > 1\geq \ve_{z_0}(\omega)$ for all $\tau\in\mathbb{D}\subset\CP^1$, and so we may apply Proposition \ref{dumppy} to see that
\[
\ve_{(z_0,0)}(\pi^*\omega + 2p^*\omega_{FS}) = \ve_{z_0}(\omega).
\]

When $t < \ve_{z_0}(\omega)$, we then get a smooth function $f$ on $X\times\CP^1$ such that $\omega_0 := \mu^*(\pi^*\omega + 2p^*\omega_{FS}) - tR_h + \ddbar f$ is a K\"ahler form. Restricting to $\mu^{-1}(X\times\overline{\mathbb{D}})$ and changing reference forms again now gives us the situation in \cite[Theorem 1.3]{CTW2}, except for the fact that the new boundary obstacle $f_0 := - t\log|s|_h - f$ isn't smooth or $\omega_0$-psh on all of $\mu^{-1}(X\times\overline{\mathbb{D}})$, as unwinding the definitions gives:
\[
\omega_0 + \ddbar f_0 = \mu^*(\pi^*\omega + 2p^*\omega_{FS}) - t\llbracket E\rrbracket.
\]
We are saved however by the fact that the obstacle is $\omega_0$-psh on the boundary -- we can use an extension trick of Boucksom \cite{Bou}, which is also outlined in \cite[Section 3]{CTW1}. Indeed, take $\chi'$ to be a compactly supported cutoff function of the origin in $\mathbb D$, such that $\chi'\equiv 1$ near 0. Then $\chi := 1-\chi'$ is equal to one in a neighborhood of $\partial\mathbb{D}$, and so $(\mu^*p^*\chi) f_0 = f_0$ near $\partial\mathbb{D}$ and $=0$ on a neighborhood of $\mu^{-1}(p^{-1}(0))\supset E$. Moreover, $\omega_0 + \ddbar((\mu^*p^*\chi)f_0)$ will fail to be positive only in the $\mathbb D$ direction and the mixed $\mathbb D$ and $X$ directions, so adding a large enough multiple of $\mu^*p^*(|\tau|^2 - 1)$ to $(\mu^*p^*\chi) f_0$ will produce a strictly $\omega_0$-psh function $\vp_0$ on all of $\ti{X}$ that agrees with $f_0$ on the boundary. We may then apply \cite[Theorem 1.3]{CTW2} to conclude regularity in this case.\\

If $t = \ve_{z_0}(\omega)$ however, we will need to use Theorem \ref{weak}. Let $\chi$ be as in the previous paragraph and define:
\begin{align*}
\alpha := \mu^*(\pi^*\omega + 2\omega_{FS}) &- tR_h - \ddbar((\mu^*p^*\chi) t\log|s|_h).
\end{align*}
Note that $\alpha$ is a smooth form on all of $\ti X$ and that it is strictly positive in a neighborhood of $\mu^{-1}(p^{-1}(\partial\mathbb D))$. We would then like to apply Theorem \ref{weak} to conclude regularity of the envelope:
\[
 \sup\left\{v\in\PSH(\ti{X}, \alpha)\ \middle|\ \limsup_{|\tau| < 1, \tau\rightarrow \ti{\tau}}v(z,\tau)\leq 0 \text{ for all }\ti{\tau}\in\partial\mathbb{D}\right\}.
\]
Now by Corollary \ref{the frog}, we have that $[\alpha]$ is big and nef on $\ti X$, so for every $\ve > 0$ we can find smooth functions $u_\ve$ such that $\alpha + \ve\ti\omega + \ddbar u_\ve$ is a K\"ahler form, where here $\ti\omega$ is a fixed K\"ahler metric on $\ti X$. Corollary \ref{the frog} also allows us to find a $\psi\in\PSH(\ti X,\alpha)$ with analytic singularities along $E_{nK}(\alpha)\subseteq \mu^{-1}(p^{-1}(0))$ such that $\alpha + \ddbar\psi$ is a K\"ahler current. Finally, we can again use the trick of Boucksom to find our $v_0$ by modifying $\psi$ near the boundary. Specifically, if $\eta$ is a cutoff function of the origin supported in $\D$ such that $\{\eta'\not=0\}\subseteq\{\chi\equiv 1\}$, then we define:
\[
v_0 := (\mu^*p^*\eta)\psi + C\mu^*p^*(|\tau|^2 - 1),
\]
for some large enough $C$ such that $v_0$ is strictly $\alpha$-psh. This works as $\alpha$ is a K\"ahler form on the set where $\eta$ is not constant. Finally, it is a standard fact that product domains such as $X\times\overline{\mathbb D}$ are Levi-flat, and in particular are thus weakly pseudoconcave. Thus, $(X\times\overline{\mathbb D}, \alpha|_{X\times\overline{\mathbb D}})$ satisfies the hypothesis of Theorem \ref{weak}, and so we are able to conclude the desired regularity of $\vp_t$.

\end{proof}

\section{Proof of Theorem \ref{weak}}

Recall we want to prove the following:
\begin{theorem}\label{weakprime}
Suppose that $(M,\omega)$ is a compact K\"ahler manifold with boundary. Suppose moreover that $\alpha$ is a closed, real $(1,1)$-form on $M$ such that the following holds:
\begin{itemize}
\item There exists a function $\psi\in\PSH(M,\alpha)$,  $\psi\leq 0$, with analytic singularities, smooth near and up to the boundary of $M$, that also satisfies $\alpha + \ddbar\psi\geq \delta\omega$ for some small $\delta> 0$.
\item For every $\ve > 0$, there exists a smooth function $u_\ve\leq -1$ such that $\alpha_\ve := \alpha + \e\omega + \ddbar u_\ve > 0$ is a K\"ahler form.
\item There exists a function $v_0\in\PSH(M,\alpha)$ such that $v_0\leq 0$, $v_0|_{\partial M}\equiv 0$, and $v_\e := \ti{\max}\{v_0,u_\e\}$ is a smooth $(\alpha+\ve\omega)$-psh function up to the boundary for all $\ve > 0$.
\end{itemize}
Then the envelope:
\begin{equation}\label{thethingprime}
V := \sup\left\{ v\in\PSH(M, \alpha)\ \middle|\ \limsup_{|z|\rightarrow \partial M} v(z)\leq 0\right\}
\end{equation}
is uniformly Lipschitz on compact subsets of $M\setminus\mathrm{Sing}(\psi)$.

If we further have that $\alpha + \ddbar v_0\geq \delta\omega$ in a neighborhood of $\partial M$ and that the boundary of $M$ is pseudoconcave, then actually $V\in C^{1,1}_\text{loc}(M\setminus\mathrm{Sing}(\psi))$.
\end{theorem}

We shall prove this using Berman's approximation technique and the estimates in \cite{CTW1}. First, we state the following well-known version of the maximum principle for sub- and super-solutions of the negative curvature K\"ahler-Einstein equation:
\begin{theorem}\label{maxprinc}
Suppose that $u, v\in\PSH(M,\omega)\cap L^\infty(M)$, and that:
\[
(\omega + \ddbar u)^n\geq e^{\beta u +f}\omega^n,\ \ (\omega + \ddbar v)^n\leq e^{\beta v + f}\omega^n
\]
\[
\limsup_{z\rightarrow z_0} u(z)\leq \liminf_{z\rightarrow z_0} v(z)\ \forall z_0\in\partial M
\]
for some number $\beta > 0$ and smooth function $f$. Then $u\leq v$ on all of $M$.
\end{theorem}

To reduce now to the case in \cite{Be}, first assume without loss of generality that $\omega\geq\alpha$, so that $\PSH(M,\alpha)\subset\PSH(M,\omega)$. Then let $h$ be the unique smooth solution to the following Dirichlet problem for the Laplacian of $\omega$:
\begin{equation}\label{h}
\Delta_\omega h = -n,\ \ \ h|_{\partial M}\equiv 0.
\end{equation}
It is then easy to see using the maximum principle for subharmonic functions that if $v$ is a candidate for the envelope \eqref{thethingprime}, then $v\leq h$, so that:
\begin{equation}
V = \sup \{v\in\PSH(M,\alpha)\ |\ v\leq h\}.
\end{equation}
Up to replacing $\alpha$ with $\alpha + \ddbar h$ (and $v_0$ with $v_0 - h$, etc.) we are then reduced to studying the envelope
\begin{equation}
V' = \sup \{v\in\PSH(M,\alpha)\ |\ v\leq 0\},
\end{equation}
and so we can now follow \cite{Be}, additionally taking into consideration the behavior at the boundary.

We begin by defining $\alpha_\ve := \alpha + \ve\omega$, so that $\alpha_\ve + \ddbar u_\ve > 0$ is a K\"ahler form on $M$. We then approximate $V'$ by the envelopes:
\[
 V_\e = \sup \{v\in\PSH(M,\alpha_\e)\ |\ v\leq 0\}.
\]
It is easy to see that the $V_\e$ are monotonically decreasing to $V'$ as $\e\rightarrow 0$, and that the convergence is actually uniform on compact sets away from $\mathrm{Sing}(\psi)$. Berman's technique is now to approximate the $V_\e$ by smooth solutions to the (non-degenerate) negative curvature K\"ahler-Einstein equation:
\begin{equation}\label{epsilonbeta}
(\alpha_\ve + \ddbar V_{\ve,\beta})^n = e^{\beta V_{\ve,\beta} + f}\omega^n,\ \ \alpha_\ve + \ddbar V_{\ve,\beta} > 0 ,\ \ V_{\ve,\beta}|_{\partial M} \equiv 0
\end{equation}
as $\beta\rightarrow\infty$, with $f = f(\e)$ a constant to be choosen momentarily, that will decrease to $-\infty$ as $\e\rightarrow 0$.

Solutions to \eqref{epsilonbeta} exist for all $\beta > 0$ by the the following theorem:
\begin{theorem}\label{CKNS}
Suppose that $(M^n,\omega)$ is a compact K\"ahler manifold with boundary, $f$ is a smooth function on $M$, $\vp$ is a smooth function on $\partial M$, and that $\beta > 0$ is a fixed constant. Then there exists a (unique) smooth solution $V$ to:
\[
(\omega + \ddbar V)^n = e^{\beta V + f}\omega^n,\ \ \omega + \ddbar V > 0,\ \ V|_{\partial M}= \vp
\]
if and only if there exists a smooth (strong) subsolution:
\[
(\omega + \ddbar v)^n \geq e^{\beta v + f}\omega^n,\ \ \omega + \ddbar v > 0,\ \ v|_{\partial M}= \vp.
\]
\end{theorem}

\begin{remark}
Theorem \ref{CKNS} is a well-known modification of the proof of \cite[Theorem A]{Bou}, which is due to several authors (\cite{Yau}, \cite{CKNS}, \cite{Gu}, \cite{Chen}, \cite{Bl2}). Indeed, if $M$ is actually a smoothly bounded domain in $\C^n$ then Theorem \ref{CKNS} was already shown in \cite{Gu} (and in \cite{CKNS} in the pseudoconvex case).
\end{remark}

We can apply Theorem \ref{CKNS} to solve \eqref{epsilonbeta} as follows. First, rewrite \eqref{epsilonbeta} as:
\begin{equation}\label{rewrite}
(\alpha_\ve + \ddbar u_\e+ \ddbar u)^n = e^{\beta u+ \beta u_\ve+ f + \log(\omega^n/(\alpha_\ve + \ddbar u_\e)^n)}(\alpha_\ve + \ddbar u_\e)^n,
\end{equation}
\[
u|_{\partial M} = -u_\ve|_{\partial M}
\]
with $u := V_{\ve,\beta} - u_\ve$. Then, by assumption, both $v_0$ and $u_\ve$ are strictly $\alpha_\e$-psh, and so their regularized maximum:
\[
v_\e := \ti{\max}\{v_0, u_\ve\}
\]
will be a smooth, strictly $\alpha_\e$-psh function  with $v_\ve\leq 0$ and $ v_\ve =v_0$ in a neighborhood of the boundary. In particular, $v_\e|_{\partial M}\equiv 0$. Choosing then:
\[
f(\ve) \leq \inf_M\log\left(\frac{(\alpha_\e + \ddbar v_\e)^n}{\omega^n}\right)
\]
makes $v_\e - u_\ve$ into a subsolution of \eqref{rewrite}, as $v_\ve\leq 0$, so that:
\begin{equation}\label{subsub}
(\alpha_\ve + \ddbar v_\e)^n \geq e^{f(\ve)}\omega^n\geq e^{\beta v_\ve+ f(\ve)}\omega^n = e^{\beta (v_\ve - u_\ve) + \beta u_\ve+ f(\ve)}\omega^n.
\end{equation}
Thus, Theorem \ref{CKNS} applies, and so the $V_{\ve,\beta}$ exist and are smooth up to the boundary. We will also assume without loss of generality that we have choosen the $f(\ve)$ such that $f \leq 0$, is strictly increasing in $\ve$, and satisfies:
\begin{equation}\label{notreallynecessary}
\lim_{\ve\rightarrow 0} e^{f(\ve)} = 0.
\end{equation}

\noindent The following is due to Berman \cite[Prop 2.4]{Be}. In it, we regard $\e$ as fixed:

\begin{proposition}\label{uniform}
Then exists constants $A, \beta_0 \geq 0$, depending on $\alpha,$ $\omega$, and $\e$ such that:
\[
\Vert V_{\ve,\beta} - V_{\ve}\Vert_{C^0(X)} \leq \frac{A \log\beta}{\beta} \text{ for all } \beta\geq\beta_0.
\]
In particular, the $V_{\ve,\beta}$ converge uniformly to $V_\e$ as $\beta\rightarrow \infty$.
\end{proposition}
\begin{proof}
Consider a point $x_0$ at which $V_{\ve}$ achieves its maximum. If $x_0\in\partial M$, then $V_{\ve}\leq 0$. If $x_0$ is in the interior of $M$, then we have that $\ddbar V_{\ve}\leq 0$ at $x_0$, so that:
\[
\sup_M\frac{(\alpha + \omega)^n}{\omega^n}\geq \frac{(\alpha_\e + \ddbar V_{\ve,\beta})^n(x_0)}{\omega^n(x_0)}\geq e^{\beta V_{\ve}(x_0) + f}.
\]
It follows that:
\begin{equation}\label{hahaha}
V_{\ve,\beta}\leq \frac{A}{\beta},
\end{equation}
where $A$ depends on $f(\e)$. Thus:
\[
V_{\ve,\beta} - \frac{A}{\beta}\leq V_\ve.
\]

For the reverse inequality, let $C > 0$ be arbitrary and consider $v\in\PSH(M,\alpha_\ve)$, $v\leq 0$. Define:
\[
g := \left(1-\frac{1}{\beta}\right) \max\{v, v_\e - C\} + \frac{1}{\beta} v_\e - \frac{n\log\beta}{\beta},
\]
for $\beta \geq 1$. One then checks that:
\[
(\alpha_\e + \ddbar g)^n \geq \frac{1}{\beta^n}(\alpha_\e + \ddbar v_\e)^n \geq  \frac{1}{\beta^n}e^{f}\omega^n = e^{\beta(-n(\log\beta)/\beta) + f}\omega^n \geq e^{\beta g + f}\omega^n,
\]
as both $v, v_\e\leq 0$. This implies that $g$ is a subsolution to \eqref{epsilonbeta}, and so Theorem \ref{maxprinc} implies
\[
g \leq V_{\ve,\beta}.
\]
Letting then $C\rightarrow +\infty$ we see:
\[
\left(1-\frac{1}{\beta}\right) v + \frac{1}{\beta} v_\e - \frac{n\log\beta}{\beta} \leq V_{\e,\beta},
\]
and so taking the supremum over all such $v$ gives:
\[
V_{\ve,\beta} - \frac{A}{\beta}\leq V_\e \leq \frac{1}{1 - 1/\beta} V_{\ve,\beta}  - \frac{v_\e}{\beta(1-1/\beta)} + \frac{1}{1-1/\beta}\frac{n\log\beta}{\beta}.
\]
Up to possibly increasing $A$, depending now on the infimum of $v_\e$, and again using the fact that $V_{\ve,\beta}\leq A/\beta$, we see:
\[
|V_{\ve,\beta} - V_\ve| \leq \frac{A\log\beta}{\beta}
\]
for all $\beta$ sufficently large, as desired.

\end{proof}

We now bound the norm of the gradient of $V_{\e,\beta}$ on the boundary, independent of $\ve$ and $\beta$. As before, note that $\PSH(M,\alpha_\e)\subset\PSH(M,2\omega)$ and that $V_{\ve,\beta}|_{\partial M}\equiv 0$, so again the maximum principle implies:
\begin{equation}\label{upp}
V_{\ve,\beta}\leq 2h.
\end{equation}
We also just checked above, \eqref{subsub}, that $v_\ve$ is a subsolution to \eqref{epsilonbeta}, so Theorem \ref{maxprinc} gives:
\[
v_\ve\leq V_{\ve,\beta}.
\]
But now, as $v_\ve$ is smooth and equal to $v_0$ near the boundary and $v_0|_{\partial M} \equiv 2h|_{\partial M}\equiv 0$ are both fixed, it follows from a simple calculus exercise that the gradient of $V_{\ve,\beta}$ is bounded in the normal direction on $\partial M$. That it is bounded in the tangent directions is obvious, and so we conclude
\begin{equation}\label{bddbdd}
\sup_{\partial M} |\nabla V_{\e,\beta}|_\omega^2 \leq C,
\end{equation}
for $C$ independent of $\ve,\beta$.

Following \cite[Prop. 4.1]{CTW1}, we now bound the gradient on the interior away from $\mathrm{Sing}(\psi)$. In order to ensure that our estimates are independent of both $\ve$ and $\beta$, we take care to make sure they are independent of $\inf f$. Note that we are now dropping the $\ve$ and $\beta$ subscripts, for ease of notation.

\begin{proposition}\label{list}
Let $V$ solve:
\begin{equation}\label{eq}
(\alpha_\e + \ddbar V)^n = e^{\beta V + f}\omega^n,\ \ \ti\omega := \alpha_\e +\ddbar V > 0,\ \ V|_{\partial M} \equiv 0.
\end{equation}
Then there exists constants $B, C, \ve_0, \beta_0 > 0$ such that the following estimate holds on $M\setminus\mathrm{Sing}(\psi)$ for all $0 < \ve \leq \ve_0$ and $\beta \geq \beta_0$:
\[
|\nabla V|^2_g \leq Ce^{-B\psi}
\]
where here $\nabla$ is the Levi-Civita connection of $\omega$, and $g$ its associated metric. $B$, $C$, $\ve_0$, and $\beta_0$ will depend only on the following quantities: $M$, $\alpha$, $\omega$, $\psi$, $\sup_M f$, and $\sup_M|\nabla f|_g^2$.
\end{proposition}
\begin{proof}
Note first that we have:
\[
\ti{V} := V - \psi \geq 0
\]
for $\e\leq \ve_0$ by the following computation. Let $C > 0$ be arbitrary. Then it follows from Demailly \cite{De-1} that:
\[
 (\alpha_\e + \ddbar\max\{\psi, v_\e - C\})^n \geq \chi_{\{\psi \geq v_\ve - C\}} (\alpha_\e + \ddbar\psi)^n + \chi_{\{v_\ve - C > \psi\}}(\alpha_\e + \ddbar v_\ve)^n
\]
\[
\geq \chi_{\{\psi\geq v_\ve - C\}} \delta^n\omega^n + \chi_{\{v_\ve - C > \psi\}} e^{f}\omega^n \geq e^{f(\ve_0)} \omega^n \geq e^{\beta \max\{\psi, v_\e - C\} + f}\omega^n,
\]
where here $\e_0$ is choosen small enough such that $e^{f(\ve_0)}\leq \delta^n$. Theorem \ref{maxprinc} then implies $\max\{\psi, v_\e - C\} \leq V$ and so letting $C\rightarrow \infty$ gives the desired inequality.

We will also define:
\[
\ti{\omega} := \alpha_\e + \ddbar V.
\]
Consider the quantity:
\[
Q := e^{h(\ti{V})} |\nabla V|_\omega^2,
\]
where here $h$ is defined by:
\[
h(s) = - Bs + \frac{1}{s- A + 1}
\]
for $A := \inf_M\ti{V}$ and $B > 0$ a constant to be choosen later. Note that $h' < 0$ and $h'' > 0$, and that:
\[
|-B\ti{V} - h(\ti{V})|\leq 1,
\]
so in particular $h(\ti{V})$ is bounded above.

We will apply the maximum principle to $Q$ -- let $x_0$ be a point at which $Q$ achieves its maximum. If $x_0\in\partial M$, then we have already seen that $|\nabla  V|_\omega^2(x_0) \leq C$, and so then:
\[
Q\leq C,
\]
on all of $M$. Using the definition of $Q$ then gives:
\[
|\nabla  V|_\omega^2\leq C e^{-h(\ti{V})}\leq C e^{-B\psi},
\]
using $(i)$, as desired.

Otherwise, assume that $x_0$ is an interior point of $M$. It cannot be that $x_0\in\mathrm{Sing}(\psi)$, as $Q\equiv 0$ on that set. Choose holomorphic normal coordinates for $\omega$ at $x_0$ that also diagonalize $\ti{\omega}$. Then we have:
\begin{equation}\label{all}
0\geq e^{-h}\Delta_{\ti\omega}Q(x_0) = |\nabla  V|_\omega^2e^{-h}\Delta_{\ti\omega}e^h + \Delta_{\ti\omega}|\nabla  V|_\omega^2 + 2e^{-h}\mathrm{Re}\langle\nabla e^h, \overline{\nabla}|\nabla V|_\omega^2\rangle_{\ti \omega}.
\end{equation}
We deal with the three terms seperately. For the first term, expanding it out gives:
\begin{equation}\label{fir}
e^{-h}\Delta_{\ti\omega} e^h \geq ((h')^2 + h'')|\nabla \ti{V}|_{\ti\omega}^2 + nh' - h'\delta\mathrm{tr}_{\ti\omega}\omega
\end{equation}
using that $h' < 0$ and:
\[
\ti\omega -\ddbar\ti{V} = \alpha + \ve\omega + \ddbar\psi\geq \delta\omega
\]
on $M\setminus\mathrm{Sing}(\psi)$. For the second term, differentiate the log of \eqref{eq} to see:
\begin{equation}\label{sec}
\Delta_{\ti\omega}|\nabla V|^2_\omega \geq \sum_{i,k}\ti{g}^{i\overline{i}}(|V_{ik}|^2 + |V_{i\overline{k}}|^2) - C|\nabla V|^2_\omega\tr_{\ti\omega}\omega + 2\beta|\nabla V|_\omega^2 + 2\text{Re}\langle\nabla f,\overline{\nabla} V\rangle_{\omega},
\end{equation}
where $C$ is a lower bound for the curvature of $\omega$ and we have assumed, without loss of generality, that $|\nabla V|_\omega^2 \geq 1$.

For the third term, following \cite[(4.8)]{CTW1}, we see that:
\begin{equation}\label{thi}
2e^{-h}\mathrm{Re}\langle\nabla e^h, \overline{\nabla}|\nabla V|_\omega^2\rangle_{\ti \omega} \geq 3h'|\nabla V|_\omega^2 + h'|\nabla\psi|_\omega^2 + \frac{C}{\delta} h'|\nabla \ti{V}|_{\ti\omega}^2
\end{equation}
\[
+ \frac{\delta}{2}h'|\nabla V|_\omega^2\tr_{\ti\omega}\omega + \sum_{i,k}\ti{g}^{i\overline{i}}|V_{ik}|^2 - (h')^2|\nabla V|_\omega^2|\nabla\ti V|^2_{\ti\omega}.
\]
Putting together \eqref{all}, \eqref{fir}, \eqref{sec}, and \eqref{thi}, and using the fact that $\psi$ has analytic singularities, we see that:
\begin{align*}
0&\geq h''|\nabla V|_\omega^2|\nabla\ti{V}|_{\ti\omega}^2 + \left(-\frac{\delta}{2}h' - C\right)|\nabla V|_\omega^2\mathrm{tr}_{\ti\omega}\omega\\ &+ (Ch' + 2\beta)|\nabla V|_\omega^2 + C h'e^{-C_1\psi} + \frac{C}{\delta}h'|\nabla \ti{V}|_{\ti\omega}^2 + 2\text{Re}\langle\nabla f,\overline{\nabla} V\rangle_{\omega},
\end{align*}
for uniform constants $C$ and $C_1$. Using Cauchy-Schwarz on the last term, we can absorb the $\nabla V$ term to get:
\begin{align*}
0&\geq h''|\nabla V|_\omega^2|\nabla\ti{V}|_{\ti\omega}^2 + \left(-\frac{\delta}{2}h' - C\right)|\nabla V|_\omega^2\mathrm{tr}_{\ti\omega}\omega\\ &+ (Ch' + 2\beta - 1)|\nabla V|_\omega^2 + C h'e^{-C_1\psi} + \frac{C}{\delta}h'|\nabla \ti{V}|_{\ti\omega}^2 - C
\end{align*}
where we have also possibly increased $C$ to be a lower bound for $|\nabla f|_\omega^2$. Choosing now $B = 2 + \max\{(2C+2)/\delta, C_1\}$ and $\beta_0 \geq C(B + 1)\geq 1$, we see that:
\[
0\geq \frac{2}{(\ti V - A + 1)^3}|\nabla V|_g^2|\nabla \ti V|_{\ti\omega}^2 + |\nabla V|_\omega^2\tr_{\ti\omega}\omega + \beta|\nabla V|_\omega^2 - Ce^{-C_1\psi} - C - C|\nabla\ti{V}|_{\ti\omega}^2.
\]
We may assume without loss of generality that $|\nabla V|_g^2\geq C(\ti V - A + 1)^3$ at $x_0$ (for if not, one may check that $Q\leq C$ for a uniform $C$ independent of $x_0$), and so we get:
\[
|\nabla V|_\omega^2 \leq \frac{C e^{-C_1\psi} + C}{\beta + \tr_{\ti\omega}\omega}.
\]
After possibly increasing $C$, we may now conclude in an analogous manner to \cite[Prop. 4.1]{CTW1}, using that:
\[
\tr_{\ti\omega}\omega \geq e^{-(\beta/n) V - f/n}\geq Ce^{-(\beta/n) V}
\]
\end{proof}

We now turn to the Hessian estimates. As before, we begin by bounding the Hessian at the boundary, independant of $\e$ and $\beta$. We largely follow the exposition of \cite{Bou}.

\begin{proposition}\label{HessB}
Suppose that $V$ is as in Proposition \ref{list}. If we additionally assume that the boundary of $M$ is weakly pseudoconcave and that $v_0$ is strictly $\alpha$-psh (i.e. $\alpha + \ddbar v_0\geq \delta \omega > 0$), then there exist constants $C, \e_0, \beta_0 > 0$ such that:
\[
|\nabla \nabla V(x)|^2_g \leq C
\]
for all $x\in\partial M$, $0 < \e \leq\e_0$, and $\beta\geq\beta_0$. As before, $\nabla$ is the Levi-Civita connection of $\omega$, $g$ is its associated metric, and $C$,$\ve_0,$ and $\beta_0$ only depend on the following quantities: $M$, $\alpha$, $\omega$, $v_0$, $\sup_M f$ and $\sup_M|\nabla f|_g^2$.
\end{proposition}
\begin{proof}

Let $x\in\partial M$ and pick coordinates centered at $x$ as in \cite[Rmk. 7.13]{Bou}. We will also use the same quantities $r, B,$ and $D_j$, $1\leq j\leq 2n$ from \cite[Lemma 7.17]{Bou}. Note that all the $D_j$'s commute.

Proposition \ref{list} implies that we have a uniform bound:
\[
\sup_{x\in\partial M} |D_j V| \leq K
\]estimates
for all $1\leq j\leq 2n$.

The tangent-tangent directions are trivial, as for $1\leq i, j\leq 2n-1$, $D_j V$ is constantly zero on $\partial M$ and $D_i$ is tangent to $\partial M$.

We now bound the tangent-normal directions -- fix $1\leq j\leq 2n-1$. We first claim that we can choose constants $\lambda$ and $\mu$ as in \cite[pg. 273]{Bou} such that:
\[
 b := (V- v_0) + \lambda h - \mu r^2 \geq 0 \text{ on } B,
\]
with equality on $\partial M$, and:
\[
 \Delta_V b \leq -\frac{\delta}{2}\tr_V\omega,
\]
where here $\Delta_V$ and $\tr_V$ are the Laplacian and trace with respect to the solution metric $\alpha_\ve + \ddbar V$. The proof is essentially the same as in \cite{Bou}, and so we omit it for brevity.

Define then:
\[
 w := K(\mu_1 b + \mu_2 |z|^2) - |D_j V|.
\]
It is easy to see that by fixing the radius of $B$ and choosing $\mu_2$, we can arrainge that $w \geq 0$ on $\partial B$. 

We now claim that $w\geq 0$ on all of $B$. To see this, let $z_0\in B$ be a minimum point for $w$ -- if $z_0\in\partial B$, then we are done, so suppose that $z_0$ is in the interior. If we have that $D_j V(z_0) = 0$, then we are also done, so suppose that $D_j V(z_0)\not= 0$ -- smoothness then implies that $D_j V$ is also non-zero in a small neighborhood of $z_0$, so that $|D_j V| = \pm D_j V$ is also smooth near $z_0$. We then compute:
\[
 \Delta_V w = K(\mu_1 \Delta_V b + \mu_2 \tr_V (\omega_{Eucl})) - \Delta_V |D_j V|
\]
\[
 \leq K\left(-\frac{\delta\mu_1}{2}\tr_V\omega + \mu_2 \tr_V (C\omega)\right) - \Delta_V |D_j V| 
\]
\[
 \leq K\left(C\mu_2 - \frac{\delta\mu_1}{2}\right)\tr_V(\omega) - \Delta_V \abs{D_j V}.
\]
One can then check that, as in \cite{Bou}, we have:
\[
\Delta_V (D_j V) =  \tr_V (D_j \ddbar V) + V_{x_n} \Delta_V a + 2a_{y_n} - 2\tr_V(da\wedge \iota_{\partial/\partial y_{n}}\alpha_\e).
\]
Applying $D_j$ to \ref{eq}, we also get:
\[
 \tr_V(D_j\ddbar V) = -\tr_V(D_j\alpha_\e) + \tr_{\omega}(D_j\omega) + D_j f + \beta D_j V,
\]
and so:
\[
\Delta_V (D_j V) = \Gamma + \beta D_j V,
\]
where:
\[
\Gamma := V_{x_n} \Delta_V a + 2a_{y_n} - 2\tr_V(da\wedge \iota_{\partial/\partial y_{n}}\alpha_\e) -\tr_V(D_j\alpha_\e) + \tr_{\omega}(D_j\omega) + D_j f.
\]
Note that we can uniformly bound $\Gamma$ on both sides:
\[
|\Gamma| \leq C\tr_V \omega,
\]
because:
\[
 \tr_V\omega \geq n \left(\frac{\omega^n}{(\alpha_\e + \ddbar V)^n}\right)^{1/n} = n \left(\frac{\omega^n}{e^{\beta V + f}\omega^n}\right)^{1/n}= n e^{-(\beta/n) V - f/n} \geq c,
\]
for $c > 0$ independant of $\e,\beta$, by \eqref{hahaha}.

We then have two cases, $D_j V(z_0) > 0$ or $D_j V(z_0) < 0$, which are basically the same. We only do the later; in that case, we have $\abs{D_j V} = -D_j V$ near $z_0$ and so:
\[
 \Delta_v w(z_0) \leq K\left(C\mu_2 - \frac{\delta\mu_1}{2}\right)\tr_V(\omega) + \Gamma + \beta D_j V(z_0) \leq -\tr_V\omega < 0
\]
for a large enough $\mu_1$, independant of $\beta$, as $D_j V (z_0) < 0$. But this is now a contradiction with the definition of $z_0$ -- as an interior minimum, we must have that $\Delta_V w(z_0) \geq 0$. We thus conclude that:
\[
 w \geq 0
\]
everywhere, as desired. This now implies:
\[
-K (\mu_1 b  + \mu_2\abs{z}^2)\leq \abs{D_j V} \leq K (\mu_1 b + \mu_2\abs{z}^2),
\]
and since both sides are $0$ at $x$, we conclude:
\[
\abs{D_{2n} D_j V}(x) \leq C,
\]
as desired. A compactness argument then makes the bound uniform over all of $\partial M$. \\

We are then left to deal with the normal-normal direction. For this, we switch back to the complex derivatives, and note that we will be done if we bound the quatity $V_{n\overline{n}}$. Also observe that, given our choice of $r$, $\{z_1,\ldots, z_{n-1}\}$ span the holomorphic tangent bundle of $M$ at $0$, which we will denote by $T^h_{\partial M}$. If we then write $\alpha_\ve$ in the $z$-coordinates as:
\[
\alpha_\ve := \sum_{j,k} \alpha_{j\overline{k}} dz^j\wedge d\overline{z}^k,
\]
we see that our bounds on the tangent-tangent and tangent-normal derivatives give:
\[
\abs{\det(\alpha_{j\overline{k}} + V_{j\overline{k}})_{1\leq j,k\leq n} - (\alpha_{n\overline{n}} + V_{n\overline{n}})\det(\alpha_{j\overline{k}} + V_{j\overline{k}})_{1\leq j, k\leq n-1}} \leq C.
\]
The Monge-Ampere equation for $V$ tells us that the first term is just equal to $e^{\beta V + f}\omega^n \leq e^C\omega^n \leq C$, by \eqref{hahaha}. Thus, it is sufficent to find a uniform lower bound for the second determinate, which will be implied by a lower bound for the form:
\[
(\alpha_\ve + \ddbar V) |_{T^h_{\partial M}} = (\alpha_\ve + \ddbar v_0)|_{T^h_{\partial M}} + \ddbar (V - v_0) |_{T^h_{\partial M}} 
\]
\[
\geq \delta\omega|_{T^h_{\partial M}} + \ddbar (V - v_0)|_{T^h_{\partial M}}.
\]
Since we have that $V - v_0$ is uniformly zero on $\partial M$, it is a standard computation to check that:
\[
\ddbar (V- v_0)|_{T^h_{\partial M}} = (\nu\cdot (V- v_0) )L_\nu,
\]
where here $\nu$ is an outward facing normal vector field on $\partial M$ and $L_\nu$ is the Levi-form of $M$ with respect to $\nu$. Since $M$ is weakly pseudoconcave, by definition we have that $L_\nu \leq 0$. But we also have that $V - v_0 \geq 0$ with equality at the boundary, so that $\nu \cdot (V - v_0) \leq 0$ also. Thus, we see that:
\[
(\alpha_\ve + \ddbar V) |_{T^h_{\partial M}} \geq \delta\omega|_{T^h_{\partial M}},
\]
as desired.

\end{proof}

Finally, we claim interior bounds on the Hessian, following \cite{CTW1}.

\begin{proposition}\label{Hess}
Suppose that $V$ is as in Proposition \ref{list}. If we additionally assume that the boundary of $M$ is weakly pseudoconcave and that $v_0$ is strictly $\alpha$-psh (i.e. $\alpha + \ddbar v_0\geq \delta \omega > 0$), then there exist constants $B,C, \e_0, \beta_0 > 0$ such that:
\[
|\nabla \nabla V|^2_g \leq Ce^{B\psi}
\]
for all $0 < \e \leq\e_0$ and $\beta\geq\beta_0$. As before, $\nabla$ is the Levi-Civita connection of $\omega$, $g$ is its associated metric, and $B$, $C$, $\ve_0,$ and $\beta_0$ only depend on the following quantities: $M$, $\alpha$, $\omega$, $v_0$, $\sup_M f$, $\sup_M|\nabla f|_g^2$, and a lower bound for $\nabla\nabla f$ with respect to $\omega$.
\end{proposition}
\begin{proof}
 The idea is the same as in \cite[Proposition 4.3]{CTW1}. First, bound the Laplacian of $V$ indpendant of $\ve,\beta$, using the standard estimates from the proof of Yau's theorem and Proposition \ref{HessB}. 
 
 We then consider the same quantity as in \cite{CTW1}, based off of bounding the perturbed first eigenvalue of $\nabla\nabla V$. As the Hessian is totally bounded on the boundary by Proposition \ref{HessB}, we only have to deal with the case of an interior maximum, and then the proof is the same as in \cite{CTW1}, the only difference being the function $f$. However, the estimates of \cite[Theorem 2.1]{CTW1} already deal with such a term, and the modifications are very minor.
\end{proof}

\begin{proof}[Proof of Theorem \ref{weakprime}]
For the first claim, Proposition \ref{list} shows that we have a uniform gradient bound on $V_{\ve,\beta}$ for all $0 < \e \leq \e_0$ and $\beta \geq \beta_0$ on any compact subset $K\subset M\setminus\mathrm{Sing}(\psi)$. Thus, we can apply Proposition \ref{uniform} to see that the same bound holds for $V_\e$, by letting $\beta\rightarrow\infty$. Letting then $\ve\rightarrow 0$ then shows that $V$ also has the same bounds, and so is uniformly Lipschitz on $K$.

If we further assume that $M$ is weakly pseudoconcave and that $v_0$ is strictly $\alpha$-psh, then Theorem \ref{Hess} shows that we actually have uniform Hessian bounds on $K$, so again letting $\beta\rightarrow\infty$ and then $\e\rightarrow 0$ implies that $V$ also has the same bounds. It then follows that $V\in C^{1,1}(K)$ by standard elliptic estimates.
\end{proof}

\begin{remark}\label{extremal}
Suppose that $M^n\subset X^n$ is a smoothly bounded domain, and let $\omega$ be a K\"ahler form on $X$. We have already seen a natural function associated to $M$ (or more precisely, to $M^c$):
\[
V_{M^c} := \sup\{v\in\PSH(X,\omega)\ |\ v\leq 0 \text{ on } M^c\}^*.
\]
It is definitional that:
\[
V_{M^c}|_{M} \leq \sup\left\{v\in\PSH(M,\omega|_M)\ \middle|\ \limsup_{z\rightarrow\partial M}v(z)\leq 0\right\} =: V,
\]
and it follows from, say, Prop. \ref{uniform}, that $V|_{\partial M} \equiv 0$. Thus, a simple gluing argument shows that we actually have equality between the two. In particular, we have re-proved that $M^c$ is {\it regular}, i.e. that $V_{M^c}$ is continuous -- see \cite[Prop. 1.6]{BBWN}, where they prove a more general condition for regularity. We actually see more however: It follows then from \cite[Theorem B]{Bou} that $V_{M^c}$ is globally Lipschitz, as we have control of the gradient on the boundary of $M$ -- thus, $M$ has the H\"older continuous property, or HCP, of order one, which should be the optimal global regularity one can expect for $V_{M^c}$ \cite[Rmk. 3.7]{Sic}. Moreover, if one assumes that $M$ is pseudoconcave (i.e. $M^c$ is pseudoconvex) then $V_{M^c}$ is actually in $C^{1,1}(M)$, by \cite{CTW2}.

It is an interesting question to ask how much one can weaken the assumption on smoothness of the boundary. We only use boundary regularity to solve the Dirichlet problems \eqref{h} and \eqref{epsilonbeta}, and, since we are only proving Lipschitz regularity, one would guess that we do not actually need either of these solutions to be smooth. Results in this direction are presumably already well-known to experts, but we were unable to find any exact references in the literature, as previous works are primarily interested in sets with low/no boundary regularity.

Finally, our Theorem \ref{weakprime} and a simple gluing argument also shows that, if we replace $\omega$ with a smooth form $\alpha$ whose class is nef and big, with $\partial M$ and $E_{nK}(\alpha)$ not intersecting, and $\alpha > 0$ in a neighborhood of $\partial M$, and further assume the existance of a $v_0$, then $M$ still has the HCP property of order one with respect to $\alpha$.
\end{remark}

\end{document}